\documentclass[reqno]{amsart}
\usepackage{amsmath,amssymb,mathrsfs,amsthm,amsfonts,mathtools}
\usepackage[inline]{enumitem} 		
\usepackage[usenames,dvipsnames]{xcolor}
\usepackage[width=\linewidth]{caption}
\definecolor{hypercolor}{HTML}{003399}
\usepackage{hyperref}
 	
\hypersetup{colorlinks=true, linkcolor=hypercolor,citecolor=hypercolor}
\usepackage[paper=letterpaper,margin=1in]{geometry}
\usepackage{newtxtext}

\newtheorem{thm}{Theorem}[section]
\newtheorem{lem}[thm]{Lemma}
\newtheorem{prop}[thm]{Proposition}
\newtheorem{cor}[thm]{Corollary}
\theoremstyle{definition}

\numberwithin{equation}{section}
\allowdisplaybreaks

\usepackage{acronym}
\acrodef{LDP}{Large Deviation Principle}
\acrodef{KPZ}{Kardar--Parisi--Zhang}

\newcommand{\e}{\varepsilon}
\newcommand{\calA}{\mathcal{A}}
\newcommand{\calT}{\mathcal{T}}
\newcommand{\calS}{\mathcal{S}}

\newcommand{\R}{\mathbb{R}}
\newcommand{\Lsp}{L}						

\newcommand{\hk}{p}							
\newcommand{\lens}{\mathsf{Lens}}			

\newcommand{\Path}{\gamma} 					
\newcommand{\geodesic}{\theta}
\newcommand{\perturbation}{\eta}

\newcommand{\dev}{\rho}						
\newcommand{\devm}{\rho^{\scriptscriptstyle\mathrm{m}}}	
\newcommand{\devlim}{\rho_*}				
\newcommand{\Zfn}{\mathsf{Z}}				
\newcommand{\Zm}{\Zfn^{\scriptscriptstyle\mathrm{m}}}
\newcommand{\hm}{h}
\newcommand{\hlim}{h_*}
\newcommand{\FK}{\mathsf{FK}}				
\newcommand{\FKm}{\FK^{\scriptscriptstyle\mathrm{m}}}

\renewcommand{\P}{\mathbb{P}}				
\newcommand{\E}{\mathbb{E}}					
\renewcommand{\d}{\mathrm{d}}				
\newcommand{\ind}{\mathbf{1}}				

\newcommand{\norm}[1]{\Vert #1\Vert}
\newcommand{\ip}[1]{\langle #1\rangle}

\renewcommand{\bar}{\overline}

\usepackage{graphicx}
\newcommand*{\Cdot}{{\raisebox{-0.5ex}{\scalebox{1.8}{$\cdot$}}}} 

\graphicspath{{./figs/}}

\title{A lower-tail limit in the weak noise theory}
\author{Yier Lin \and Li-Cheng Tsai}
\address[Yier Lin]{Departments of Statistics, University of Chicago}
\email{ylin10@uchicago.edu}
\address[Li-Cheng Tsai]{Departments of Mathematics, University of Utah}
\email{lctsai.math@gmail.com}

\begin{document}
\maketitle

\begin{abstract}
We consider the variational problem associated with the Freidlin--Wentzell Large Deviation Principle of the Stochastic Heat Equation (SHE).
The logarithm of the minimizer of the variational problem gives the most probable shape of the solution of the Kardar--Parisi--Zhang equation conditioned on achieving certain unlikely values.
Taking the SHE with the delta initial condition and conditioning the value of its solution at the origin at a later time, under suitable scaling, we prove that the logarithm of the minimizer converges to an explicit function as we tune the value of the conditioning to $ 0 $.
Our result confirms the physics prediction \cite{kolokolov09,meerson16,kamenev16}.
\end{abstract}

\section{Introduction}
In this paper we study the following variational problem.
For a given $ \dev = \dev(t,x) \in L^2([0,2]\times\R) $, consider the heat equation driven by the potential $ \dev $, with the delta initial condition:
\begin{align}
	\label{e.he}
	&&
	&\partial_t \Zfn = \tfrac12 \partial_{xx} \Zfn + \dev \Zfn,
	\
	(t,x) \in (0,2)\times\R,
	&&
	\Zfn(0,\Cdot) = \delta(\Cdot).&
\end{align}
We write $ \Zfn=\Zfn[\dev] = \Zfn[\dev](t,x) $ for the solution of this equation.
The variational problem of interest is
\begin{align}
	\label{e.variationalproblem}
	\inf\big\{ \tfrac{1}{2} \norm{ \dev }^2_2 \, : \, \Zfn[\dev](2,0) =  e^{-\lambda} \big\},
\end{align}
where $ \lambda >0 $ is a parameter, and $ \norm{\Cdot}_2 $ denotes the $ L^2 $ norm.
It was proven in \cite[Corollary~2.5(b)]{tsai22wnt} that, for all large enough $ \lambda $, the variational problem~\eqref{e.variationalproblem} has a unique minimizer $ \devm=\devm(\lambda;t,x) $.

This variational problem describes the Freidlin--Wentzell Large Deviation Principle (LDP) of the Stochastic Heat Equation (SHE) and the Kardar--Parisi--Zhang (KPZ) equation.   
Consider the SHE $ \partial_t Z_\e = \frac12 \partial_{xx} Z_\e + \sqrt{\e} \xi Z_\e $ with $ Z_\e(0,\Cdot)=\delta(\Cdot) $, where $ \xi $ denotes the spacetime white noise.
It was proven in \cite{lin21} that $ Z_\e $ enjoys the LDP with speed $ \e^{-1} $ and the rate function
$ 
	I_\mathrm{SHE}(f) := \inf\{ \frac{1}{2} \norm{ \dev }^2_2 : \Zfn[\dev] = f \}.
$
Given the LDP, we see that the variational problem \eqref{e.variationalproblem} corresponds to conditioning the value of $ Z_\e(2,0) $ around $ e^{-\lambda} $.
Accordingly, the function $ \Zfn[\devm](t,x) $ is the most probable path of $ Z_\e $ under the conditioning.
The solution of the SHE produces the solution of the KPZ equation through taking logarithm. 
Namely, $ \log Z_\e = h_\e $, and $ h_\e $ solves the KPZ equation.
Hence $ \hm := \log \Zfn[\devm] $ gives the most probable path of the solution of the KPZ equation, which we refer to as the \textbf{most probable shape}.

Studying the Freidlin--Wentzell LDPs for the SHE and the KPZ equation through the variational problem goes under the name of the \emph{weak noise theory}.
There has been much development around the weak noise theory in the physics and mathematics literature 
\cite{kolokolov07,kolokolov08,kolokolov09,janas16,kamenev16,meerson17,hartmann18,meerson18,smith18exact,smith18finitesize,asida19,hartmann19,smith19,hartmann20,hartmann21,gaudreaulamarre21,lin21}, and more recently around the connection of the weak noise theory to integrable PDEs \cite{krajenbrink20det,krajenbrink21,krajenbrink21flat,bettelheim22,krajenbrink22,mallick22,tsai22wnt}.
Among the many questions of interest in the weak noise theory are the behaviors of the most probable shape under certain limits.
Particular limits of interest are sending the conditioned value $ \log\Zfn[\dev](2,0) $ to $ +\infty $ or $ -\infty $.
We refer to them as the (one-point) upper- and lower-tail limits, respectively.
For the delta initial condition considered here, the upper- and lower-tail limits of the most probability shape (under suitable scaling) have been predicted in the physics works \cite{kolokolov09,meerson16,kamenev16}.
Later, the work \cite{gaudreaulamarre21} gave a rigorous proof of the upper-tail limit.
The behaviors of the upper- and lower-tail limits are very different, and a rigorous proof of the lower-tail limit remained open.

In this paper, we prove the lower-tail limit of the most probable shape.
Let us introduce the scaling.
Set $ \Zfn_\lambda[\dev](t,x) := \Zfn[\lambda\dev(\Cdot,\lambda^{-1/2}\Cdot)](t,x\lambda^{1/2}) $.
Under this scaling, the variational problem becomes
\begin{align}
	\label{e.variationalproblem.scaled}
	\eqref{e.variationalproblem}
	=
	\lambda^{5/2} \, \inf\big\{ \tfrac{1}{2} \norm{ \dev }^2_2 \, : \, \Zfn_\lambda[\dev](2,0) = e^{-\lambda} \big\},
\end{align}
and its minimizer is $ \devm_\lambda(t,x) := \lambda^{-1}\devm(\lambda; t,\lambda^{1/2}x) $.
Let $ \Zfn_\lambda[\devm_\lambda](t,x) := \Zm_\lambda(t,x) $ and $ \hm_\lambda (t,x) := \frac{1}{\lambda}\log \Zm_\lambda(t,x) $.
The main result is as follows.
\begin{thm}
\label{t.main}
Let $ \hlim $ be given by \eqref{e.devlim} and \eqref{e.hlim.geodesic.0} and depicted in Figure~\ref{f.hlim} in Section~\ref{s.notation.tools}.
For any $ \delta>0 $,
\begin{align*}
	\lim_{\lambda\to\infty} \, 
	\sup\big\{
		| \hm_{\lambda}(t,x) - \hlim(t,x) | 
		\, : \,
		(t,x) \in (\delta,2]\times[-\delta^{-1},\delta^{-1}]
	\big\}
	=0.
\end{align*}
\end{thm} 

Let us provide a context of Theorem~\ref{t.main} in terms of Hamilton--Jacobi-Fokker--Planck (HJ-FP) equations.
To begin, one can derive the Euler--Lagrangian equation for the variational problem~\eqref{e.variationalproblem.scaled} and turn the equation into a system of Hamilton equations.
This derivation is done in the physics literature of the weak noise theory (see \cite[Appendix]{kamenev16}, \cite[Supp.\ Mat.\ A]{meerson16}, \cite[Supp.\ Mat.\ A]{krajenbrink21}) and has recently been proven in \cite[Theorem 2.1]{tsai22wnt} (at the level of $ \Zm_\lambda $).
At the level of $ \hm_\lambda $, the Hamilton equations are
\begin{subequations}
\label{e.HJ-FK}
\begin{align}
	\label{e.HJ-FK.1}
	\partial_t \hm_\lambda &= \tfrac{1}{2\lambda} \partial_{xx} \hm_\lambda + \tfrac12 (\partial_x \hm_\lambda)^2 + \devm_\lambda,
\\
	-\partial_t \devm_\lambda &= \tfrac{1}{2\lambda} \partial_{xx} \devm_\lambda - \partial_x( \devm_\lambda \partial_x \hm_\lambda ),
\end{align}
\end{subequations}
where the second equation is solved backward in time with a suitable terminal condition at $ t=2 $.
This system can be viewed as an instance of the HJ-FP equations studied in mean field games \cite{lions07,cardaliaguet10,gueant11}.
The lower-tail limit $ \lambda\to\infty $ here corresponds to the inviscid limit in the language of mean field games and PDE.
We expect that, for $ \devm_\lambda \leq 0 $, the solution of \eqref{e.HJ-FK} converges to the entropy solution of
\begin{subequations}
\label{e.HJ-FK.limit}
\begin{align}
	\partial_t \hlim &= \tfrac12 (\partial_x \hlim)^2 + \devlim,
\\
	-\partial_t \devlim &= - \partial_x( \devlim \partial_x \hlim ).
\end{align}
\end{subequations}
The physic works \cite{kolokolov09,meerson16,kamenev16} solved \eqref{e.HJ-FK.limit} for the entropy solution and found an explicit expression for $ \devlim $.
Accordingly, the function $ \hlim $ can be expressed in terms of $ \devlim $; see Section~\ref{s.notation.tools.devlimhlim}.

From the PDE perspective, the challenge of proving Theorem~\ref{t.main} lies in controlling $ \devm_\lambda $.
We will show in Appendix~\ref{s.a.devm<=0} that $ \devm_\lambda \leq 0 $ for all $ \lambda $ large enough, and we will explain in Section~\ref{s.notation.tools.devlimhlim} that the results from \cite{lin21} gives $ \devm_\lambda \to \devlim $ in $ L^2 $.
These properties alone, however, do \emph{not} suffice for $ \hm_\lambda \to \hlim $.
For $ b>0 $, set $ w_\lambda(t,x) := \devlim(t,x) \ind_{\{|x|>\lambda^{-b}\}} $, which is non-positive and converges to $ \devlim $ in $ L^2 $.
It is possible to show that, there exist a small enough $ b $ and a neighborhood $ \Omega $ around $ (2,0) $ such that $ \lim_{\lambda\to\infty}\inf_{\Omega}|\frac{1}{\lambda}\log\Zfn_\lambda[w_\lambda] - \hlim| > 0 $.
In many settings of mean field games, the term $ \devm_\lambda $ in \eqref{e.HJ-FK.1} is replaced by a better-behaved term, and phenomena like the one just shown does not occur. 

Our proof proceeds through the Feynman--Kac formula and bypasses the need to control $ \devm_\lambda $.
The first key observation is that the property $ \devm_\lambda\to\devlim $ in $ L^2 $ alone \emph{does} suffice for proving the lower bound $ (\liminf_{\lambda\to\infty} \hm_\lambda) \geq \hlim $ and in fact a stronger version of it: Proposition~\ref{p.lwbd}.
Roughly speaking, Proposition~\ref{p.lwbd} states that the change of $ \hm_\lambda $ along a geodesic (defined in Section~\ref{s.notation.tools.devlimhlim}) is bounded from below by the change of $ \hlim $ along the same geodesic.
The second key observation is that the upper bound, which is the more subtle bound, follows by combining Proposition~\ref{p.lwbd} and the property $ \Zm_\lambda(2,0) = e^{\lambda\hm_\lambda(2,0)}=e^{-\lambda} $, as well as a H\"{o}lder continuity estimate.
This observation is manifested in the proof in Section~\ref{p.upbd.lens}.

Let us mention that $ \Zm_\lambda $ permits an explicit formula.
The expression was derived by \cite{krajenbrink21} based on \cite{krajenbrink20det} at a physics level of rigor, and later proven in \cite[Theorem~2.3, Corollary~2.5]{tsai22wnt}.
Extracting the limit of $ \hm_\lambda := \frac{1}{\lambda} \log\Zm_\lambda $ from this explicit formula is an interesting open problem and will provide another proof of Theorem~\ref{t.main}.

We conclude this introduction by discussing some related works.
There exists another approach to study the Freidlin--Wentzell LDPs for the SHE and the KPZ equation, based on explicit formulas of the one-point distribution.
The analysis has been carried out in the physics literature for various initial and boundary conditions \cite{ledoussal16short,krajenbrink17short,krajenbrink18half}.
While the explicit formulas does not give direct access to the most probable shape, they allow for studying LDPs for the KPZ equation in the long time regime \cite{ledoussal16long,sasorov17,corwin18,krajenbrink18simple,krajenbrink18systematic,krajenbrink19thesis,krajenbrink19,ledoussal19kp,corwin20lower,corwin20general,das21,kim21,lin21half,cafasso22,ghosal20,tsai22}. See also \cite{ganguly22,liu22} for related work.

\subsection*{Outline}
In Section~\ref{s.notation.tools} we prepare some notation and recall some basic tools.
In Section~\ref{s.lwbd}, we prove a lower bound, which in particular gives the lower half of Theorem~\ref{t.main}.
In Section~\ref{s.holder}, we prove a spatial H\"{o}lder continuity.
In Section~\ref{s.upbd}, we prove the upper half of Theorem~\ref{t.main}.

\subsection*{Acknowledgment}
We thank Daniel Lacker and Panagiotis Souganidis for useful discussions.
The research of Tsai was partially supported by the NSF through DMS-2243112 and the Sloan Foundation Fellowship.

\section{Notation and tools}
\label{s.notation.tools}

\subsection{The Feynman--Kac formula}
Let us introduce the Feynman--Kac formula for $ \Zfn_\lambda[\dev] $.
For $ s\leq t \in[0,2] $, let $ \mathrm{BB}_\lambda((t,x)\to(s,y)) $ denote the law of the following Brownian bridge, which goes \emph{backward} in time:
\begin{align}
	\label{e.bb}
	\lambda^{-1/2} \big( B(t-u) - \tfrac{u-s}{t-s}B(t-s) \big) + \tfrac{(u-s)x+(t-u)y}{t-s}, \qquad u \in [s,t],
\end{align}
where $ B $ denotes a standard Brownian motion.
For any $ \dev \in L^2([0,2]\times\R) $, define the Feynman--Kac expectation
\begin{align}
	\label{e.FKexp}
	&\FK_\lambda[\dev](s,y;t,x)
	:=
	\E\Big[ \exp\Big( \int_s^t \lambda \dev(u,W_\lambda(u))\, \d u \Big) \Big],&
	&
	W_\lambda \sim \mathrm{BB}_\lambda((t,x)\to(s,y)).&
\end{align}
Let $ \hk_\lambda(t,x) := \exp(-\lambda\frac{x^2}{2t})\sqrt{\lambda/(2\pi t)} $ denote the scaled heat kernel.
We have the Feynman--Kac formula
\begin{align}
	\label{e.FKformula}
	\Zfn_\lambda[\dev](t,x) &= \hk_\lambda(t,x) \, \FK_\lambda[\dev](0,0;t,x),
\\
	\label{e.FKformula.stot}
	\Zfn_\lambda[\dev](t,x) &= \int_{\R} \hk_\lambda(t-s,x-y) \, \FK_\lambda[\dev](s,y;t,x) \, \Zfn_\lambda[\dev](s,y) \, \d y,
	\qquad
	s<t\in(0,2].
\end{align}
For the sake of completeness, we provide a proof of \eqref{e.FKformula} in Appendix~\ref{s.a.FKformula}.
The formula~\eqref{e.FKformula.stot} follows from \eqref{e.FKformula} via conditioning $ W_\lambda $ at time $ s $.

\subsection{The limits $ \devlim $ and $ \hlim $}
\label{s.notation.tools.devlimhlim}
We begin by introducing $ \devlim $, which is the limit of $ \devm_\lambda $ as $ \lambda\to\infty $.
Let $ r $ be the unique $ C^1[1,2) $-valued solution with $ r|_{(1,2)} > \frac{\pi}{2} $ of the differential equation
\begin{align}
	\label{e.r.ode}
	\dot{r} = 2^{1/2} \pi^{-1/2} r^2 (r-\tfrac{\pi}{2})^{1/2},\qquad r(1) = \tfrac{\pi}{2},
\end{align}
and symmetrically extend $ r $ to $ C^1(0,2) $ by setting $ r(t) : = r(2-t) $ for $ t\in(0,1) $.
Set $ \ell(t):= 1/r(t) $.
Integrating the differential equation by separation of variables and analyzing the result near  $ t=2 $ give
\begin{align}
	\label{e.r.bound}
	r(t) = r(2-t) &= (\tfrac{2}{3})^{2/3}(\tfrac{\pi}{2})^{1/3} t^{-2/3} + o(t^{-2/3})  \text{ when } t \to 0.
\end{align}
In particular, $ r\to+\infty $ as $ t\to 0 $ or $ 2 $.
We hence set $ \ell(0) =\ell(2) := 0 $ to make $ \ell \in C[0,2] $.
The limit $ \devlim $ is given by
\begin{align}
	\label{e.devlim}
	\devlim(t,x) := - \tfrac{1}{2\pi} r(t) \big( 1 - \tfrac{x^2}{\ell(t)^2} \big)_+.
\end{align}
This expression was derived in the physics works \cite{kolokolov09,meerson16,kamenev16} by solving \eqref{e.HJ-FK.limit} for the entropy solution.
In the mathematics literature, \cite{lin21} gives the $ L^2 $ convergence
\begin{align}
	\label{e.devm.L2convergence}
	\lim_{\lambda\to\infty} \norm{ \devm_\lambda - \devlim }_2
	=
	0.
\end{align}
More precisely, this follows by combining $ \norm{\devm_\lambda-\devlim}^2_2 = \norm{\devm_\lambda}^2_2 - \norm{\devlim}^2_2 + 2\ip{\devlim,\devlim-\devm_\lambda} $ and Equations~(4.28) and (4.32) in \cite{lin21}.
Note that the expression within the limsup in \cite[Equation~(4.28)]{lin21} is $ \frac{1}{2}\norm{\devm_\lambda}^2_2 $.

We now turn to $ \hlim $. The function is defined by
\begin{align}
	\label{e.hlim}
	\hlim(t,x) := - \inf\Big\{ \int_0^t \frac12\dot{\Path}^2(s) - \devlim(s,\Path(s)) \, \d s \, : \, \Path \in H^1((0,0)\to(t,x)) \Big\},
\end{align}
with the convention $ \inf\emptyset := \infty $.
Hereafter $ H^1((s,y)\to(t,x)) $ denotes the space of paths $ \Path: [s,t] \to \R $ such that $ (\Path(s),\Path(t))=(y,x) $ and $ \int_s^t \dot{\Path}^2 \, \d u < \infty $.
To understand the motivation of this definition, recall that $ \hm_\lambda := \frac{1}{\lambda} \log \Zfn_\lambda[\devm_\lambda] $.
In view of \eqref{e.devm.L2convergence}, it is natural to expect (but non-trivial to prove) that this function approximates $ \frac{1}{\lambda} \log \Zfn[\devlim] $ as $ \lambda\to\infty $.
Since $ \devlim $ is uniformly continuous except near $ (0,0) $ and $ (2,0) $, through the Feynman--Kac formula~\eqref{e.FKformula}, the last expression can be analyzed by Varadhan's lemma, which yields \eqref{e.hlim} as $ \lambda\to\infty $.

Crucial to our analysis is the notion of geodesics.
Proposition~4.3 in \cite{lin21} shows that the infimum is achieved in $ H^1((0,0)\to(t,x)) $ for all $ (t,x)\in(0,2]\times\R $.
We call a path realizing the infimum a \textbf{geodesic}.
Proposition~4.3 in \cite{lin21} characterizes all geodesics, and they are depicted in Figure~\ref{f.geodesics}; see the caption there.
We write $ (s,y) \xrightarrow{\scriptscriptstyle\mathrm{geod.}} (t,x) $ for ``$ (s,y) $ and $ (t,x) $ are connected by a geodesic with $ s\leq t \in[0,2] $'', and write $ (s,y) \xrightarrow{\scriptscriptstyle\mathrm{geod.}\geodesic} (t,x) $ to signify that the geodesic is $ \geodesic $.
We can rewrite \eqref{e.hlim} more explicitly as
\begin{align}
	\label{e.hlim.geodesic.0}
	&&&\hlim(t,x) 
	= 
	\int_0^t \Big( -\frac12\dot{\geodesic}^2(u) + \devlim(u,\geodesic(u)) \Big) \, \d u,
	&
	&
	(0,0) \xrightarrow{\scriptscriptstyle\mathrm{geod.}\geodesic} (t,x).&
\end{align}
Note that, for any $ (t,x) \neq (2,0) $ with $ t>0 $, there exists a unique geodesic that connects $ (0,0) $ and $ (t,x) $; when $ (t,x) = (2,0) $, the expression~\eqref{e.hlim.geodesic.0} holds for $ \geodesic = a \ell $, for all $ a\in[-1,1] $; when $ t=0 $, by definition $ \hlim(0,x) = 0\ind_{\{x=0\}} - \infty \ind_{\{x\neq 0\}} $; a plot of $ \hlim $ is shown in Figure~\ref{f.hlim}.
Further, for any given $ (s,y) \xrightarrow{\scriptscriptstyle\mathrm{geod.}\geodesic} (t,x) $, the geodesic $ \geodesic $ can always be extended backward in time to $ (0,0) $.
This together with \eqref{e.hlim.geodesic.0} gives
\begin{align}
	\label{e.hlim.geodesic}
	&&&\hlim(t,x) - \hlim(s,y)
	=
	\int_s^t \Big( -\frac12\dot{\geodesic}^2(u) + \devlim(u,\geodesic(u)) \Big) \, \d u,
	&&
	(s,y) \xrightarrow{\scriptscriptstyle\mathrm{geod.}\geodesic} (t,x).&
\end{align}

\begin{figure}
\begin{minipage}[b]{.48\linewidth}
\includegraphics[width=\linewidth]{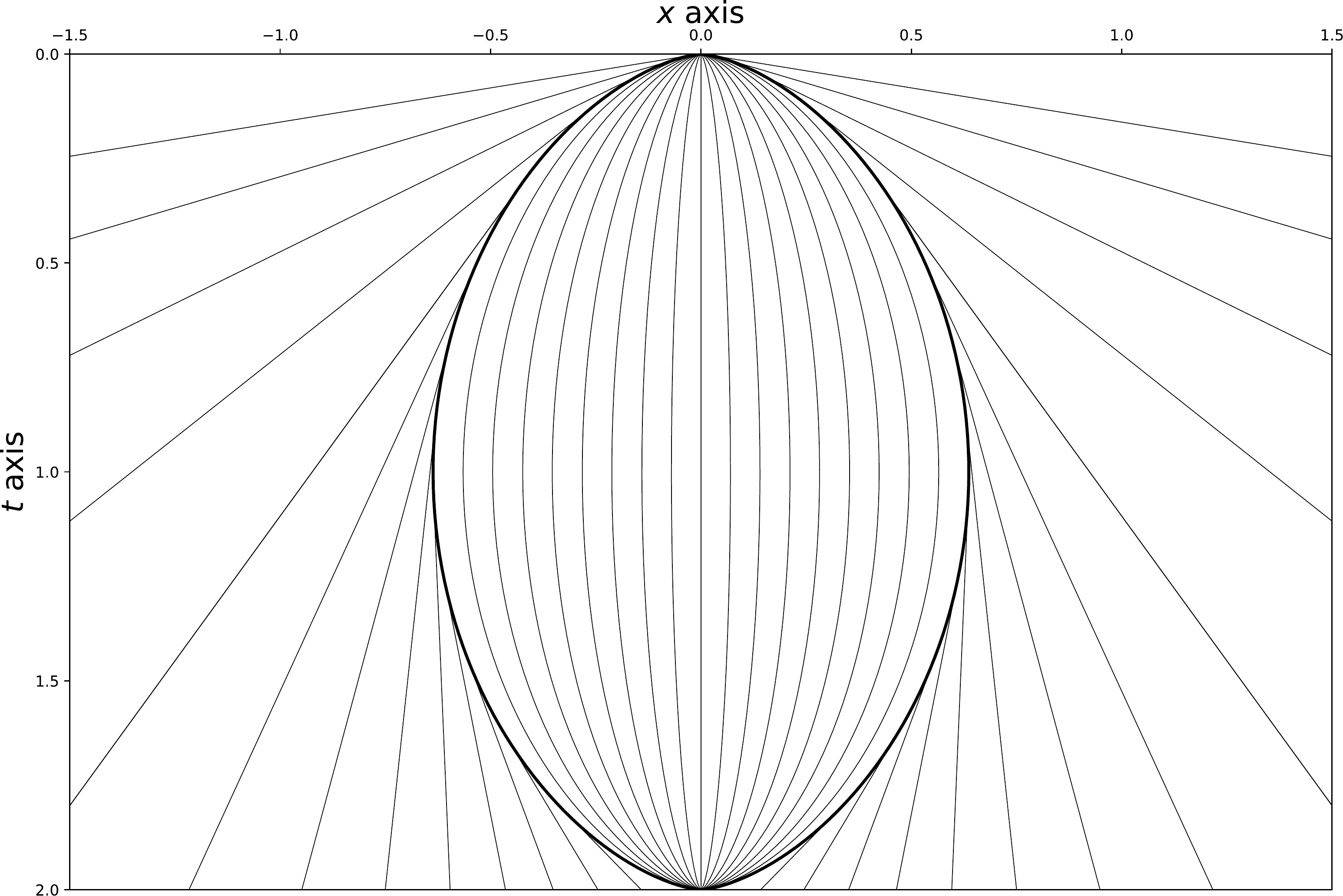}
\caption{The geodesics. Note that the $ t $ axis goes top-down. The thick lines are $ \pm \ell $. We refer to the region bounded by $ \pm\ell $ the lens region. The geodesics within the lens region are $ a\ell $, for $ a\in[-1,1] $. The geodesics outside of the lens region are straight lines that touch $ \pm\ell $ at tangent.}
\label{f.geodesics}
\end{minipage}
\hfill
\begin{minipage}[b]{.48\linewidth}
\includegraphics[width=\linewidth]{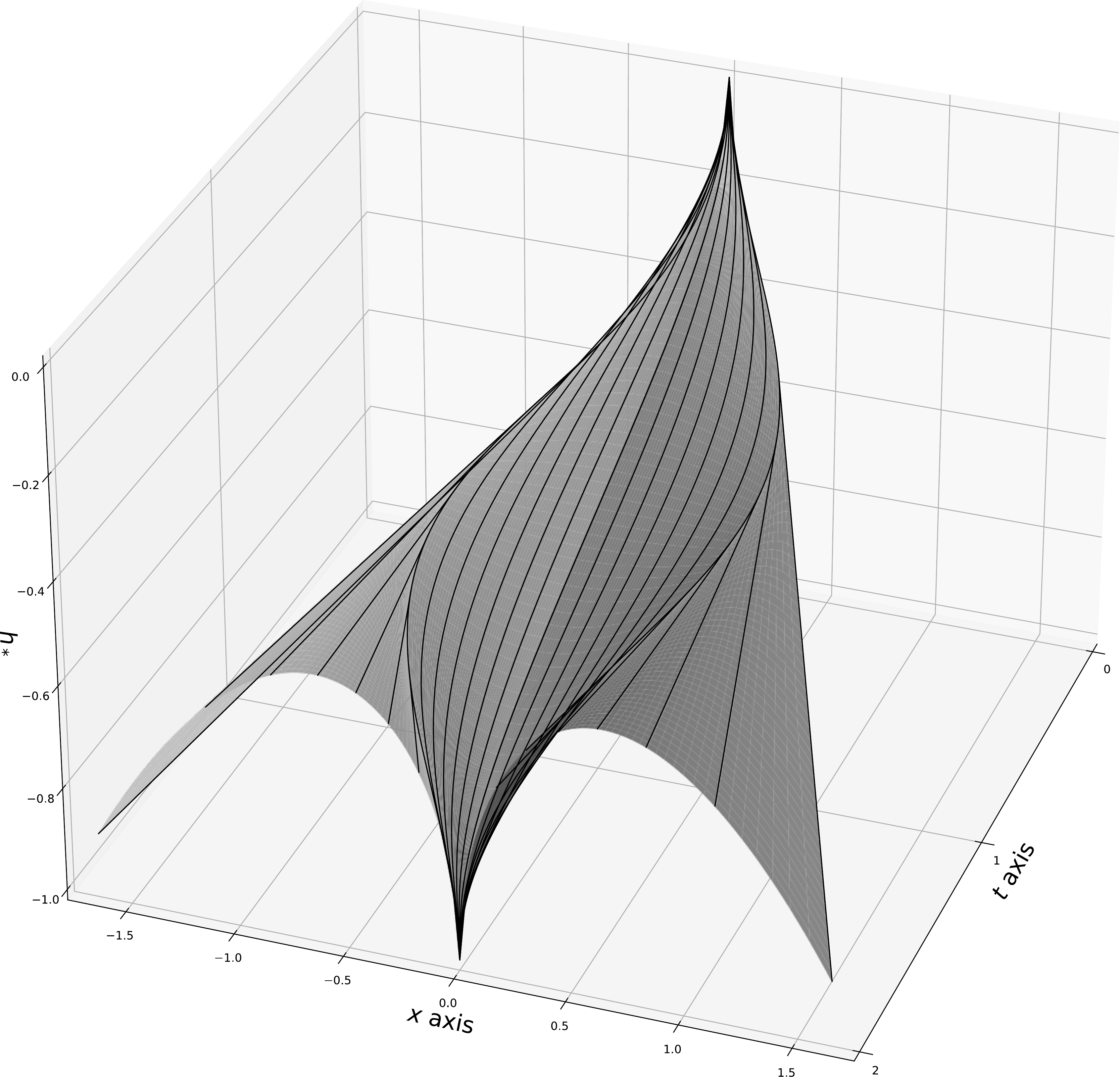}
\caption{The plot of $ \hlim $. The curves on $ \hlim $ are traced out by the geodesics. For better visualization, we did not show some part of $ \hlim $ outside of the lens region.}
\label{f.hlim}
\end{minipage}
\end{figure}

\section{Lower bound}
\label{s.lwbd}

In this section we prove the following proposition.
\begin{prop}
\label{p.lwbd}
For any $ \{ \dev^{(\lambda)} \}_{\lambda} \subset L^2([0,2]\times\R) $ with $ \dev^{(\lambda)}\to\devlim $ in $ L^2 $,
\begin{align*}
	\liminf_{\lambda\to\infty}\,
	\inf\Big\{  
		\frac{1}{\lambda} \log \FK_\lambda[\dev^{(\lambda)}](s,y;t,x) - \Big( \hlim (t,x) - \hlim(s,y) + \frac{(y-x)^2}{2(t-s)} \Big)  
	\Big\}
	\geq 0,
\end{align*}
where the infimum goes over all $ (s,y) \xrightarrow{\scriptscriptstyle\mathrm{geod.}} (t,x) $.
\end{prop}

Proposition~\ref{p.lwbd} already gives the lower half of Theorem~\ref{t.main}.
To see how, set $ (s,y) = (0,0) $ and $ \dev^{(\lambda)}=\devm_\lambda $ in Proposition~\ref{p.lwbd}, note that $ \hlim(0,0):=0 $ and that any $ (t,x)\in(0,2]\times\R $ is connected to $ (0,0) $ through a geodesic, and combine the result with \eqref{e.FKformula} for $ \dev=\devm_\lambda $ (note that $ \Zfn_\lambda[\devm_\lambda] =: \exp(\lambda\hm_\lambda) $).
Doing so gives the following.
\begin{cor}[The lower half of Theorem~\ref{t.main}]
\label{c.main.lwbd}
For any $ \delta>0 $,
\begin{align*}
	\liminf_{\lambda\to\infty} \ \inf\Big\{ \hm_\lambda(t,x) - \hlim(t,x) \, : (t,x) \in [\delta,2]\times[-\delta^{-1},\delta^{-1}] \Big\} \geq 0.
\end{align*}
\end{cor}

\begin{proof}[Proof of Proposition~\ref{p.lwbd}]
We start by deriving an inequality:
For any $ \dev\in\Lsp^2([0,2]\times\R) $ and $ \Path \in H^1((s,y)\to(t,x)) $,
\begin{align}
	\label{e.girsanov+jensen}
	\frac{1}{\lambda} \log \FK_\lambda[\dev](s,y;t,x)
	\geq
	\int_s^{t} \E\big[ \dev(t-u, (W^0_\lambda-\Path)(u)) \big] \, \d u
	-
	\frac{1}{2} \int_s^t \dot{\Path}^2 \, \d u
	+
	\frac{(x-y)^2}{2(t-s)},
\end{align}
where $ W^0_\lambda = \mathrm{BB}_\lambda((t,0)\to(s,0)) $, which goes \emph{backward} in time; see \eqref{e.bb}.
To derive this inequality, take \eqref{e.FKexp} and express the Brownian bridge as in \eqref{e.bb}.
Let $ \Path_\text{tilted}(u) := \Path(u) - \frac{1}{t-s}((u-s)x+(t-u)y) $.
Apply the Cameron--Martin--Girsanov theorem to transform the measure of the standard Brownian motion $ B $ to that of $ B+\lambda^{1/2}\Path_\text{tilted} $:
\begin{align*}
	\FK_\lambda[\dev](s,y;t,x)
	=
	\E\Big[ \exp\Big( 
		\int_s^t \lambda \dev(u,(W^0_\lambda-\Path)(u))\, \d u 
		+ \lambda^{1/2} \int_s^t \dot{\Path}_\text{tilted}\, \d B 
		- \frac{\lambda}{2} \int_s^{t} \dot{\Path}_\text{tilted}^2 \,\d u 
	\Big) \Big].
\end{align*}
Next, apply Jensen's inequality $ \frac{1}{\lambda} \log\E[\exp(\ldots)] \geq \frac{1}{\lambda}\E[(\ldots)] $ to the right side. 
Note that $ \E[\int_{s}^t \dot{\Path}_\text{tilted} \,\d B] = 0 $ and that $ \frac12 \int_s^t \dot{\Path}^2_\text{tilted} \,\d u = \frac12 \int_s^t \dot{\Path}^2 \,\d u - (x-y)^2/(2(t-s)) $.
We see that \eqref{e.girsanov+jensen} follows.

Let us explain the idea of the rest of the proof.
Fix $ (s,y) \xrightarrow{\scriptscriptstyle\mathrm{geod.}\geodesic} (t,x) $.
In \eqref{e.girsanov+jensen}, set $ \Path\mapsto-\geodesic $ and $ \dev \mapsto \dev^{(\lambda)} $:
\begin{align}
	\label{e.girsanov+jensen.1}
	\frac{1}{\lambda} \log \FK_\lambda[\dev^{(\lambda)}](s,y;t,x)
	\geq
	\int_s^{t} \E\big[ \dev^{(\lambda)}(u, (\geodesic+W^0_\lambda)(u)) \big] \, \d u
	-
	\frac{1}{2} \int_s^t \dot{\geodesic}^2 \, \d u
	+
	\frac{(x-y)^2}{2(t-s)}.
\end{align}
On the right side, only the first term depends on $ \lambda $.
To guess what the limit of it could be, recall that $ \dev^{(\lambda)} \to\devlim $ in $ L^2 $, and note that $ W^0_\lambda $ converges in law to the constant-zero path as $ \lambda\to\infty $.
We hence expect the limit to be $ \int_s^t \devlim(u,\geodesic(u))\, \d u $.
Once the first term on the right side of \eqref{e.girsanov+jensen.1} is replaced by $ \int_s^t \devlim(u,\geodesic(u))\, \d u $, 
by \eqref{e.hlim.geodesic}, the entire right side becomes $ \hlim(t,x) - \hlim(s,y) + (x-y)^2/(2(t-s)) $, which is the desired lower bound.
Hence it suffices to pass the first term on the right hand of \eqref{e.girsanov+jensen.1} to the limit.
\emph{However}, taking such a limit seems challenging, because we only have $ \dev^{(\lambda)} \to \devlim $ in $ L^2 $, not in $ L^\infty $.
To circumvent this challenge, we perturb $ \geodesic $ and average over the perturbation:
\begin{align}
	\label{e.perturbation}
	&\geodesic_a(u) := \geodesic(u) + a \perturbation(u),
	&&
	\perturbation(u) := (u-s)^{3/4}\ind_{[s,(t+s)/2]}(u) + (t-u)^{3/4}\ind_{((t+s)/2,t]}(u).&
\end{align}
In \eqref{e.girsanov+jensen}, substitute in $ \Path\mapsto-\geodesic_a $ and $ \dev \mapsto \dev^{(\lambda)} $. 
For small $ a_0>0 $, apply $ \frac{1}{a_0} \int_0^{a_0}(\Cdot)\, \d a $ to the result.
We have
\begin{align}
	\label{e.girsanov+jensen.2}
	\frac{1}{\lambda} \log \FK_\lambda[\dev^{(\lambda)}](s,y;t,x)
	\geq
	\frac{1}{a_0} \int_0^{a_0}\hspace{-6pt}\int_s^{t} \E\big[ \dev^{(\lambda)}(u, (\geodesic_a+W^0_\lambda)(u)) \big] \d u \d a 
	-
	\frac{1}{2a_0} \int_0^{a_0}\hspace{-6pt}\int_s^t \dot{\geodesic}_a^2 \d u \d a 
	+
	\frac{(x-y)^2}{2(t-s)}.
\end{align}
Compared with \eqref{e.girsanov+jensen.1}, the right side of \eqref{e.girsanov+jensen.2} is better behaved.
We will take $ \lambda\to\infty $ first and $ a_0\to 0 $ later.

Below, we will show that, uniformly over $ (s,y) \xrightarrow{\scriptscriptstyle\mathrm{geod.}} (t,x) $,
\begin{align}
	\label{e.perturb.control1}
	\lim_{\lambda\to\infty} \frac{1}{a_0} &\int_s^{t} \int_0^{a_0}\E\big[ \big| \dev^{(\lambda)}(u, (\geodesic_a+W^0_\lambda)(u))-\devlim(u, (\geodesic_a+W^0_\lambda)(u)) \big| \big] \, \d a \d u = 0,
\\
	\label{e.perturb.control2}
	\lim_{(a_0,\lambda)\to(0,\infty)} \frac{1}{a_0} &\int_s^{t} \int_0^{a_0}\E\big[ \big| \devlim(u, (\geodesic_a+W^0_\lambda)(u)) - \devlim(u,\geodesic(u)) \big| \big] \, \d a \d u = 0,
\\
	\label{e.perturb.control3}
	\lim_{a_0\to 0} \Big(\frac{1}{2a_0} &\int_0^{a_0}\int_s^t \dot{\geodesic}_a^2 \, \d u \d a  - \frac{1}{2} \int_s^t \dot{\geodesic}^2 \, \d u \Big) = 0.
\end{align}
Once these limits are obtained, taking the iterated limit $ \lim_{a_0\to 0} \lim_{\lambda\to\infty} $ on both sides of \eqref{e.girsanov+jensen.2} and combining the result with \eqref{e.hlim.geodesic} completes the proof.

To prove \eqref{e.perturb.control1}, consider the expression within the limit, swap the integrals and the expectation, and perform a change of variables $ a\mapsto z := \geodesic(u) + W^0_\lambda(u) + a \perturbation(u) $.
Doing so gives
\begin{align*}
	\E\Big[ \frac{1}{a_0}  \int_s^{t} \frac{1}{\perturbation(u)}\int_{(\geodesic+W^0_\lambda)(u)}^{{(\geodesic+W^0_\lambda+a_0\perturbation)(u)}}  |\dev^{(\lambda)}-\devlim|(u,z) \,\d z \d u \Big].
\end{align*}
Apply the Cauchy--Schwarz inequality over the double integral to bound the expression by
$
	\E[ \norm{\dev^{(\lambda)}-\devlim}_{2} (\frac{1}{a_0} \int_s^t \frac{1}{\perturbation}\,\d u )^{1/2} ].
$
The expression within the last expectation is deterministic; the last integral is uniformly bounded over $ s<t\in[0,2] $ and does not depend on $ x,y $. 
Sending $ \lambda\to\infty $ using $ \norm{\dev^{(\lambda)}-\devlim}_{2} \to 0 $ gives the desired result~\eqref{e.perturb.control1}.

To prove \eqref{e.perturb.control2}, fix a small $ \delta>0 $ and divide the $ u $ integral into $ [s,t]\cap([0,\delta)\cup(2-\delta,2]) $ and $ [s,t]\cap[\delta,2-\delta] $.
For the former, by \eqref{e.devlim.bd}, the contribution is bounced by $ c\, \delta^{1/3} $.
For the latter, note that the function $ \devlim $ is uniformly Lipschitz on $ [\delta,2-\delta]\times\R $. 
The contribution of the latter is hence bounded by $ \int_s^t c(\delta) \,(|a_0|+ \E|W^0_\lambda(u)|) \, \d u $, where $ c(\delta)<\infty $ depends only on $ \delta $.
For any fixed $ \delta>0 $, the last integral converges to zero as $ {(a_0,\lambda)\to(0,\infty)} $ uniformly over $ s \leq t\in[0,2] $, and does not depend on $ x,y $. 
The desired result~\eqref{e.perturb.control2} follows.

Move onto \eqref{e.perturb.control3}. 
Let $ t_\text{exit} := \inf\{ u : |\geodesic(u)| > |\ell(u)| \} $ be the time that $ \geodesic $ exits the lens region, with the convention $ \inf\emptyset = +\infty $; see Figure~\ref{f.geodesics}.
Expand the expression within the limit in \eqref{e.perturb.control3} and evaluate the integrals over $ a $ to get
\begin{align*}
	\frac{a_0^2}{6} \int_s^t \dot{\perturbation}^2 \, \d u 
	+ 
	\frac{a_0}{2}\int_{[s,t]\cap[0,t_\text{exit}]} \dot{\perturbation} \, \dot{\geodesic} \, \d u 
	+
	\frac{a_0}{2} \int_{[s,t] \cap[t_\text{exit},2]} \dot{\perturbation} \,\dot{\geodesic} \, \d u .
\end{align*}
Given the prefactors, which converge to $ 0 $ as $ a_0\to 0 $, it suffices to argue that the integrals are uniformly (in $ \geodesic $) bounded.
This holds for the first integral since it is independent of $ \geodesic $.
For the second integral, we have $ |\dot{\geodesic}(u)| \leq |\dot{\ell}(u)| $.
This and \eqref{e.elldot.bd} shows the uniform boundedness.
For the third integral, the derivative $ \dot{\geodesic} $ is a constant $ \pm\dot{\ell}(t_\text{exit}) $; see Figure~\ref{f.geodesics}.
The third integral is hence $ \pm\dot{\ell}(t_\text{exit})(\perturbation|_{\partial( [s,t] \cap[t_\text{exit},2] )}) $.
Using \eqref{e.elldot.bd} and \eqref{e.perturbation}, it is not hard to show the uniform boundedness.
The desired result~\eqref{e.perturb.control3} follows.
\end{proof}

\section{Spatial H\"{o}lder continuity}
\label{s.holder}

\begin{prop}
\label{p.holder}
For any $ \delta>0 $, there exists a constant $ c(\delta)\in(0,\infty) $ such that, for all $ (t,x),(t,x') \in [\delta,2]\times[-\delta^{-1},\delta^{-1}] $ and for all $ \lambda \geq \frac{1}{c(\delta)} $, we have
$
	| \hm_\lambda(t,x) - \hm_\lambda(t,x') | \leq c(\delta) |x-x'|^{\frac{2}{13}-\delta}.
$
\end{prop}

\begin{proof}
We begin by describing the structure of the proof.
By the Feynman--Kac formula~\eqref{e.FKformula}, it suffices to show that, for all $ (t,x),(t,x') \in [\delta,2]\times[-\delta^{-1},\delta^{-1}] $,
\begin{align}
	\label{e.p.holder.goal}
	\FKm_\lambda(0,0;t,x)
	:=
	\E\Big[ \exp\Big( \int_0^t \lambda\devm_\lambda(u,W^{x}_\lambda(u))\, \d u \Big) \Big]
	\geq 
	\FKm_\lambda(0,0;t,x') \frac{\hk_\lambda(t,x')}{\hk_\lambda(t,x)} \cdot e^{-\lambda c(\delta)|x-x'|^{\frac{2}{13}-\delta}},
\end{align}
where $ W^x_\lambda \sim \mathrm{BB}((t,x)\to(0,0)) $ and $ W^{x'}_\lambda \sim \mathrm{BB}((t,x')\to(0,0)) $.
Once this is done, taking the logarithm on both sides and swapping $ x $ and $ x' $ give the desired result.
To prove \eqref{e.p.holder.goal}, we will start from the left side, and perform a sequence of ``surgeries''.
The goal is to convert the expectation on the left side to the expectation on the right side.
Each surgery will reduce the expression by a factor, and we will need to control the factor.

The first surgery is a truncation in time.
Let $ s \in (0,t) $ be close to $ t $, which will be fixed later, and consider  
$
	\calA_1
	:= 
	\{ 
		\int_{s}^{t} |\devm_\lambda( u, W^{x}_\lambda (u)) | \, \d u \leq  (t-s)^{b_1}
	\}. 
$
Hereafter, our analysis will involve powers of $ (t-s) $ and $ |x-x'| $.
The exponents $ b_1,b_2,\ldots $ will all be $ \in (0,1) $ and be fixed later.
Within the expectation on the left side of \eqref{e.p.holder.goal}, introduce the indicator $ \ind_{\calA} $.
Doing so makes the expectation smaller. 
Given the indicator, the integral from $ s $ to $ t $ is at least $ -(t-s)^{b_1} $.
Hence
\begin{align}
	\tag{Surgery I}
	\label{e.p.holder.step1}
	&\text{left side of \eqref{e.p.holder.goal}}
	\geq
	e^{-\lambda (t-s)^{b_1}}
	I_1,&
	&
	I_1:= \E\Big[ \ind_{\calA} \exp\Big( \int_{0}^s \lambda\devm_\lambda (u, W^{x}_\lambda(u))\, \d u \Big) \Big].&
\end{align}

We proceed to the second surgery, which is applied to $ I_1 $.
Within the expectation in \eqref{e.p.holder.step1}, condition on $ W^{x}_\lambda(s) =y $.
Under such a conditioning, the two segments of the bridge $ \{W^{x}_\lambda(u)\}_{u\in[s,t]} $ and $ \{W^{x}_\lambda(u)\}_{u\in[0,s]} $ are independent, and their laws are respectively $ \mathrm{BB}((t,x)\to(s,y)) $ and $ \mathrm{BB}((s,y)\to (0,0)) $.
Let
\begin{align}
	\label{e.surgery.f}
	&f(\lambda,y) := \P\Big[  \int_{s}^{t} \big|\devm_\lambda (u, \bar{W}_\lambda(u)\big) \big| \, \d u \leq  (t-s)^{b_2} \, \Big],
	&&
	\bar{W}_\lambda \sim \mathrm{BB}((t,x)\to(s,y)) .&
\end{align}
We have
$
	I_1 = \E[ f(\lambda,W^{x}_\lambda(s))\, \FKm_\lambda(0,0;s,W^{x}_\lambda(s)) ].
$

To complete the second surgery, we need a lower bound on $ f(\lambda,z) $ that is uniform in $ z\in\R $.
In \eqref{e.surgery.f}, express the bridge as a Brownian motion (similar to \eqref{e.bb}) and apply the Cameron--Martin--Girsanov theorem.
We have, for any $ \Path \in H^1((t,x)\to(s,y)) $,
\begin{align}
	\label{e.surgery.f2.control1}
	f(\lambda,y) 
	= 
	\E\Big[ e^{- \frac{\lambda}{2} \int_s^{t} \dot{\Path}^2 \d u} e^{ \lambda^{1/2} \int_{s}^{t} \dot{\Path} \d B} \ind_{\calA_2} \Big],
	&&
	\calA_2 := \Big\{\int_s^{t} |\devm_\lambda (u, \bar{W}_{\lambda} - \Path)| \d u \leq (t-s)^{b_1} \Big\}.
\end{align}
Let $ \perturbation(u) := (u-s)^{b_2}\ind_{[s,(t+s)/2]}(u) + (t-u)^{b_2}\ind_{((t+s)/2,t]}(u) $ and set $ \Path = a \perturbation $, for $ a\in\R $.
The first exponential factor in \eqref{e.surgery.f2.control1} is at least $ \exp(-c \lambda \,(t-s)^{2b_2-1}) $. 
To control the second exponential factor, introduce the indicator $ \ind_{\calA_3} $, where
$ 
	\calA_3 := \{ |\int_{s}^{t} \dot{\Path} \d B| \leq \lambda^{1/2} (t-s)^{b_3} \}.
$
We have
$
	f(\lambda,y) 
	\geq 
	\exp(-\lambda\,c\,((t-s)^{2b_2-1}+(t-s)^{b_3}))
	\P[ \calA_2 \cap \calA_3 ].
$
For the last factor, use $ \P[ \calA_2\cap\calA_3 ] \geq 1 - \P[\calA_2^\mathrm{c}] - \P[\calA_3^\mathrm{c}] $ and use Markov's inequality to bound
\begin{align*}
 	\P[\calA_2^\mathrm{c}] &\leq (t-s)^{-b_1} \int_s^{t} \E\big[\big|\devm_\lambda (u, \bar{W}_{\lambda} - a\perturbation)\big| \big] \d u,
\\
	\P[\calA_3^\mathrm{c}] &\leq \big(\lambda^{1/2} (t-s)^{b_3}\big)^{-2} \E\Big[ \Big|\int_s^t \dot\Path \d B \Big|^2 \Big] \leq \frac{c}{\lambda}(t-s)^{2b_2-2b_3-1}.
\end{align*}
Putting these inequalities together and average the result over $ a\in[0,1] $.
We have
\begin{align*}
	f(\lambda,y) 
	\geq 
	e^{-\lambda c\,((t-s)^{2b_2-1}+(t-s)^{b_3})}
	\Big( 1 - \frac{c}{\lambda}(t-s)^{2b_2-2b_3-1} - (t-s)^{-b_1} \int_0^1 \int_s^{t} \E\big[\big|\devm_\lambda (u, \bar{W}_{\lambda} - a\perturbation)\big| \big] \d u \d a  \Big).
\end{align*}
For the last term, apply the same argument in the proof of \eqref{e.perturb.control1}:
Swap the integrals with the expectation, perform a change of variables $ a\mapsto y := \bar{W}_{\lambda} + a\perturbation $, and apply the Cauchy--Schwarz inequality over the integrals.
We bound the term by $ c\, \norm{\devm_\lambda}_2 (t-s)^{(1-b_2)/2} \leq c\, (t-s)^{(1-b_2)/2} $.
This gives a lower bound of $ f(\lambda,y) $.
Recall the expression of $ I_1 $ from after \eqref{e.surgery.f}.
We have now completed the second surgery 
\begin{align}
	\tag{Surgery II}
	\label{e.p.holder.step2}
	I_1 
	\geq 
	e^{-\lambda c\,((t-s)^{2b_2-1}+(t-s)^{b_3})}
	\big( 1 - \big(\tfrac{c}{\lambda}(t-s)^{2b_2-2b_3-1} + c\,(t-s)^{\frac12-b_1-\frac{b_2}{2}}\big) \big)_+ \, I_2,
\end{align}
where $ I_2 := \E[ \FKm_\lambda(0,0;s,W^{x}_\lambda(s)) ] $.

Before moving on to the third surgery, let us fix the exponents and combine \eqref{e.p.holder.step1} and \eqref{e.p.holder.step2}. 
Set $ b_1=\frac{1}{11} $, $ b_2=\frac{7}{11} $, and $ b_3 = \frac{1}{11} $.
Assume $ |t-s| $ is small enough so that in \eqref{e.p.holder.step2} we have $ (1-(\ldots))_+ \geq \exp(-c\,(\ldots)) $.
Using this in \eqref{e.p.holder.step2} and combining the result with \eqref{e.p.holder.step1} give
\begin{align}
	\tag{Surgery I+II}
	\label{e.p.holder.step12}
	&\text{(left side of \eqref{e.p.holder.goal})}
	\geq 
	e^{-\lambda c\,(t-s)^{\frac{1}{11}}}I_2,&
	&
	I_2 := \E[ \FKm_\lambda(0,0;s,W^{x}_\lambda(s)) ].
\end{align}

In the third and last surgery, we will convert $ I_2 $ to the right side of \eqref{e.p.holder.goal}.
The first step is to convert the $ W^{x}_\lambda(s) $ in $ I_2 $ into $ W^{x'}_\lambda(s) $.
Let $ q^x(y) $ and $ q^{x'}(y) $ denote their respective probability density functions.
Write $ I_2 $ as the integral $ \int_{\R} \FKm_\lambda(0,0;s,y) q^x(y)\, \d y $.
Note that $ q^x(y) $ and $ q^{x'}(y) $ can be expressed in terms of the heat kernel.
Doing so gives $ q^x(y) = \frac{\hk_\lambda(x')}{\hk_\lambda(x)} \exp( - \lambda\frac{|x-x'|^2}{2(t-s)} ) \exp( -\lambda\frac{(x'-y)(x-x'))}{(t-s)} ) q^{x'}(y) $.
To control the last exponential factor, we restrict the integral to $ |y-x'| \leq \frac{1}{2} |x-x'|^{b_4} $.
Doing so makes the resulting integral smaller.
We have
\begin{align*}
	I_2 
	\geq 
	\frac{\hk_\lambda(x')}{\hk_\lambda(x)} e^{ - \lambda\frac{|x-x'|^2+|x-x'|^{1+b_4}}{2(t-s)} } 
	\E\big[ \FKm_\lambda(0,0;s,W^{x'}_\lambda(s)) \ind\{ |W^{x'}_\lambda(s)-x'| \leq \tfrac{1}{2}|x-x'|^{b_4} \} \big].
\end{align*}
Note that the ratio of heat kernels is expected since our goal is \eqref{e.p.holder.goal}.
In the last expectation, write $ \ind\{\ldots\} = 1-\ind\{\ldots\} $ and decompose the expectation into two accordingly.
The first expectation is equal to $ \E[ \exp( \int_{0}^{s} \lambda \devm_\lambda(u,W^{x'}_\lambda(u))\, \d u ) ] $.
Granted that $ \devm_\lambda \leq 0 $ (Appendix~\ref{s.a.devm<=0}), we extend the integral from $ [0,s] $ to $ [0,t] $, which only makes the expectation smaller.
The result is exactly $ \FKm_\lambda(0,0;t,x') $.
For the latter expectation, use $ \FKm_\lambda \leq 1 $, which follows from \eqref{e.FKexp} and $ \devm_\lambda \leq 0 $, and use $ \P[ |W^{x'}_\lambda(s)-x'| > |x-x'|^{+b_4} ] \leq \exp(-\lambda\frac{|x-x'|^{2b_4}}{c(\delta)(t-s)}) $, for all $ x,x'\in[-\delta^{-1},\delta^{-1}] $.
We have
\begin{align*}
	I_2 
	\geq 
	\frac{\hk_\lambda(x')}{\hk_\lambda(x)} e^{ - \lambda\frac{|x-x'|^2+|x-x'|^{1+b_4}}{2(t-s)} } 
	\big( \FKm_\lambda(0,0;t,x') - e^{-\lambda\frac{|x-x'|^{2b_4}}{c(\delta)(t-s)}} \big)_+.
\end{align*}
Combining Proposition~\ref{p.lwbd} and \eqref{e.hlim.bd} shows that $ \FKm_\lambda(0,0;s,y) \geq e^{-\lambda c} $.
Equipped with this bound, we factor out the first term within the last $ (\ldots)_+ $.
What remains is bounded below by $ (1-\exp(-\lambda\frac{|x-x'|^{2b_4}}{c(\delta)(t-s)}+\lambda c))_+ $.
This concludes the third surgery
\begin{align}
	\tag{Surgery III}
	\label{e.p.holder.step3}
	I_2
	\geq 
	 e^{ - \lambda\frac{|x-x'|^2+|x-x'|^{1+b_4}}{2(t-s)} } 
	\big( 1-e^{-\lambda\frac{|x-x'|^{2b_4}}{c(\delta)(t-s)}+\lambda c} \big)_+ \,
	\frac{\hk_\lambda(x')}{\hk_\lambda(x)}\FKm_\lambda(0,0;t,x').
\end{align}

It remains only to fix $ s $ and $ b_4 $.
Set $ (t-s) = |x-x'|^{\frac{22}{13}} $ and $ b_4 = \frac{11}{13}-\delta $.
We assume that $ |x-x'| $ is small enough so that the $ +\lambda c $ term in \eqref{e.p.holder.step3} is subdominant to its preceding term.
Note that this assumption does not lose any generality because the interval $ [-\delta^{-1},\delta^{-1}] $ is compact.
Also, this assumption is consistent with the previous assumption that $ (t-s) $ is small enough.
The prefactors in \eqref{e.p.holder.step12} and \eqref{e.p.holder.step3} are $ \geq \exp(-c(\delta)|x-x'|^{\frac{2}{13}-\delta}) $.
This completes the proof.
\end{proof}

\section{Upper bound}
\label{s.upbd}

Our goal here is to prove the following proposition, which together with Corollary~\ref{c.main.lwbd} gives Theorem~\ref{t.main}.

\begin{prop}[Upper half of Theorem~\ref{t.main}]
\label{p.upbd}
For any $ \delta>0 $,
\begin{align*}
	\limsup_{\lambda\to\infty} \ \sup\Big\{ \hm_\lambda(t,x) - \hlim(t,x) \, : (t,x) \in [\delta,2]\times[-\delta^{-1},\delta^{-1}] \Big\} \leq 0.
\end{align*}
\end{prop}

We will prove this proposition separately within the lens region and outside of the lens region.
More precisely, the lens region is $ \lens := \{(t,x) : |x|\leq \ell(t) \} $.
We will prove the pointwise version of Proposition~\ref{p.upbd} and argue in Section~\ref{s.upbd.uniform} how to upgrade the pointwise result to a uniform result.

\subsection{Pointwise bound within $ \lens $}
\label{p.upbd.lens}
Fixing any $ (t_0,x_0) \in \lens $, we seek to prove 
\begin{align}
	\label{e.upbd.lens.ptwise}
	\limsup_{\lambda\to\infty} \hm_\lambda(t_0,x_0) \leq \hlim(t_0,x_0).
\end{align}
Given the H\"{o}lder continuity from Proposition~\ref{p.holder}, it suffices to consider $ (t_0,x_0) \in \lens^\circ $, which we assume hereafter.
To begin the proof, in \eqref{e.FKformula.stot}, set $ \dev=\devm_\lambda $, $ (t,x) = (2,0) $, and $ s=t_0 $ and write $ \Zfn_\lambda[\devm_\lambda] := e^{\lambda\hm_\lambda} $.
Restrict the integral to a small interval $ [-v+x_0,x_0+v] $.
Doing so makes the integral smaller, whereby 
\begin{align*}
	e^{\lambda\hm_\lambda(2,0)}
	\geq
	\int_{[-v+x_0,x_0+v]} \hk_\lambda(2-t_0,y) \FK_\lambda[\dev](t_0,y;2,0) \, e^{\lambda\hm_\lambda(t_0,y)} \, \d y.
\end{align*}
Given that $ (t_0,x_0) \in \lens^\circ $, we assume $ v $ is small enough so that each $ y $ in this integral lies in $ \lens $.
For such a $ y $, there exists a geodesic that connects $ (t_0,y) $ to $ (2,0) $, so by Proposition~\ref{p.lwbd}, $ \FK_\lambda[\dev](t_0,y;2,0) \geq $$ \exp(\lambda \hlim(2,0)) $ $ \exp( -\lambda\hlim(t_0,y)) $ $ \exp(\frac{\lambda y^2}{2(2-t_0)}) $ $\exp(-o(\lambda)) $.
Note that $ \hlim(2,0) $ and $ \hm_\lambda(2,0) $ are both equal to $ -1 $, so we can cancel these two factors from both sides.
The factor $ \exp(\frac{\lambda y^2}{2(2-t_0)}) $ can be combined with $ \hk_\lambda(2-t_0,y) $.
Altogether we have
\begin{align*}
	1
	\geq
	\sqrt{\lambda/(2\pi(2-t_0))}
	\int_{[-v+x_0,x_0+v]} e^{\lambda(\hm_\lambda(t_0,y)-\hlim(t_0,y) - o(1))} \, \d y.
\end{align*}
By Proposition~\ref{p.holder} the function $ \hm_\lambda(t_0,y) $ is H\"{o}lder in $ y $, and it is not hard to verify from \eqref{e.hlim} that $ \hlim(t,x) $ is Lipschitz except near $ t=0 $.
Hence the integral is bounded below by $ 2v \exp(\lambda(\hm_\lambda(t_0,x_0)-\hlim(t_0,x_0) - o(1))) \exp(-\lambda c v^{1/13}) $.
Apply $ \frac{1}{\lambda}\log(\Cdot) $ to both sides, take $ \lambda\to\infty $, and take $ v\to 0 $ later.
Doing so gives the desired result~\eqref{e.upbd.lens.ptwise}.

\subsection{From a pointwise bound to a uniform bound}
\label{s.upbd.uniform}
Here we explain how to upgrade a pointwise lower bound into a uniform one.
First, thanks to the H\"{o}lder continuity from Proposition~\ref{p.holder}, we can upgrade \eqref{e.upbd.lens.ptwise} into
\begin{align*}
	&\limsup_{\lambda\to\infty} \{ \hm_\lambda(t_0,x) - \hlim(t_0,x) : |x| \leq \ell(t_0)\} \leq 0,
	&&
	\text{for any fixed } t_0\in(0,2].&
\end{align*}
In fact, the H\"{o}lder continuity allows us to go slightly beyond the lens:
\begin{align}
	\label{e.upbd.lens.spatial}
	&\limsup_{\lambda\to\infty} \big\{ \hm_\lambda(t_0,x) - \hlim(t_0,x) \, : \, |x| \leq \ell(t_0) + v \big\} \leq c v^{1/13},
	&&
	\text{for any fixed } t_0\in(0,2] \text{ and } v>0.&
\end{align}
To further upgrade \eqref{e.upbd.lens.spatial} into a spacetime-uniform bound, we will develop a one-sided time continuity estimate in Lemma~\ref{l.holder.t}.
Combining Lemma~\ref{l.holder.t} and \eqref{e.upbd.lens.spatial} gives the desired uniform bound:
\begin{align}
	\label{e.upbd.lens.uniform}
	\limsup_{\lambda\to\infty} \big\{ \hm_\lambda(t,x) - \hlim(t,x) \, : \, (t,x)\in \lens\cap([\delta,2]\times[-\delta^{-1},\delta^{-1}]) \big\} \leq 0.
\end{align}

\begin{lem}
\label{l.holder.t}
For any $ \delta>0 $,
\begin{align*}
	\liminf_{(\lambda,u)\to(\infty,0)} \, \inf\big\{ \hm_\lambda(s,x) - \hm_\lambda(t,x) : s<t\in[\delta,2], \, |t-s|\leq u, \, |x|\leq \delta^{-1} \big\} \geq 0.
\end{align*}
\end{lem}
\begin{proof}
In \eqref{e.FKformula.stot}, set $ \dev=\devm_\lambda $ and bound $ \FKm_\lambda(s,y;t,x) \leq 1 $, which holds because $ \devm_\lambda\leq 0 $ (Appendix~\ref{s.a.devm<=0}).
Recognize the resulting integral as $ \E[ \Zm_\lambda(s,x+Y) ] $, where $ Y $ is a zero-mean Gaussian with variance $ (t-s)/\lambda $.
Divide this expectation into $ \{|Y|\leq v\} $ and $ \{|Y| > v\} $.
For the former, use Proposition~\ref{p.holder} to bound expectation from above by $ \exp(\lambda (\hm_\lambda(s,t) + c v^{1/13})) $.
For the latter, use $ \Zm_\lambda(s,y) = \hk_\lambda(s,y) \FKm_\lambda(s,y;t,x) \leq \hk_\lambda(s,y) \leq c(\delta)\sqrt{\lambda} $ to bound the expectation from above by $ c(\delta)\sqrt{\lambda} \exp(-\lambda v^2/(3(t-s))) $.
Altogether, we have
\begin{align}
	\label{e.l.holder.t.1}
	e^{\lambda \hm_\lambda(t,x)}
	\leq
	e^{\lambda (\hm_\lambda(s,x) + c v^{1/13})}
	+
	c(\delta) \sqrt{\lambda} e^{-\lambda \frac{v^2}{3(t-s)}}.
\end{align}
Recall the lower bound of $ \hm_\lambda $ from Corollary~\ref{c.main.lwbd}, and note that $ \hlim $ is bounded on $ [\delta,2]\times[-\delta^{-1},\delta^{-1}] $; see \eqref{e.hlim.bd}.
Given this bound, we see that for fixed $ \delta,v>0 $, the first term on the right side of \eqref{e.l.holder.t.1} dominate the entire right side as $ (\lambda,u)\to(\infty,0) $, uniformly over $ s<t\in[\delta,2] $ and $ |x|\leq\delta^{-1} $ with $ |t-s| \leq u $.
This being the case, apply $ \log(\Cdot) $ to both sides, take the relevant infimum, take $ \liminf_{(\lambda,u)\to(\infty,0)} $, and finally take $ v\to 0 $.
This completes the proof.
\end{proof}

\subsection{Bound within $ \lens^\mathrm{c} $}
To complete the proof of Proposition~\ref{p.lwbd}, we establish the bound in $ \lens^\mathrm{c} $.
We will only carry out the proof of the pointwise bound: The same argument in Section~\ref{s.upbd.uniform} applies to upgrade the pointwise bound to the uniform bound.
Fix $ (t_0,x_0) \in \lens^\mathrm{c} $ with $ t_0\in(0,2] $.
Assume $ x_0 > \ell(t_0) $ to simplify notation.

The first step is to decompose the Feynman--Kac formula according to the first hitting time at the boundary of $ \lens $.
Express $ \Zm_\lambda(t_0,x_0) $ by \eqref{e.FKformula}.
The relevant Brownian bridge $ W_\lambda \sim \mathrm{BB}((t_0,x_0)\to (0,0)) $ goes backward in time.
Let $ \tau := \sup\{ u\in[0,t_0] : W_\lambda(u) \leq \ell(u) \} $ be the first (going backward in time) hitting time  at the boundary of $ \lens $.
In the Feynman--Kac expectation (namely \eqref{e.FKexp} with $ \dev=\devm_\lambda $ and $ (s,y;t,x) = (0,0;t_0,x_0) $), decompose the integral into $ \int_0^\tau $ and $ \int_\tau^{t_0} $ and bound the contribution of the latter using $ \devm_\lambda \leq 0 $.
For the former, note that the law of $ \{W_\lambda(u)\}_{u\in[0,\tau]} $ conditioned on $ \tau=s $ is $ \mathrm{BB}_\lambda((s,\ell(s))\to (0,0)) $.
We recognize the resulting expectation as $ \E[ ( \Zm_\lambda/\hk_\lambda)(\tau,\ell(\tau)) ] $.
Hence,
\begin{align}
	\label{e.upbd.lensc.FK}
	e^{\lambda\hm_\lambda(t_0,x_0)}
	=:
	\Zm_\lambda(t_0,x_0)
	\leq
	\hk_\lambda(t_0,x_0) \, \E\big[ \big( \Zm_\lambda/\hk_\lambda \big)(\tau,\ell(\tau)) \big].
\end{align}

To control the right side of \eqref{e.upbd.lensc.FK}, we discretize in time.
Let $ s_i := \frac{i}{n} t_0 $, and decompose the expectation in \eqref{e.upbd.lensc.FK} into $ n $ expectations according to $ \tau\in[0,s_1) $, $ \tau\in[s_1,s_2) $, $ \ldots $.
For the first one, we use $ (\Zm_\lambda/\hk_\lambda)(t,x) = \FKm_\lambda(0,0;t,x) \leq 1 $.
For the remaining $ (n-1) $ ones, use \eqref{e.upbd.lens.uniform} to bound $ \Zm_\lambda(u,\ell(u)) $ by $ \exp(-\lambda(\hlim(u,\ell(u))+o(1))) $, where the $ o(1) $ depends only on $ t_0/n $.
Given this bound, we further bound $ ( \Zm_\lambda/\hk_\lambda)(\tau,\ell(\tau)) $ by $ c(n,t_0) \max_{u\in[s_i,s_{i+1})} \exp(\lambda(\hlim(u,\ell(u))+\frac{\ell(u)^2}{2u}+o(1))) $.
Altogether, we have
\begin{align}
	\label{e.upbd.lensc.discrete}
	e^{\lambda\hm_\lambda(t_0,x_0)}
	\leq
	c(n,t_0) e^{-\frac{\lambda x_0^2}{2t_0}}
	\Big( 
		\P\big[\tau\in[0,s_1)\big] + \sum_{i=1}^{n-1} \max_{u\in[s_i,s_{i+1})}\Big\{ e^{ \lambda (\hlim(u,\ell(u))+ \frac{\ell(u)^2}{2u} + o(1) ) } \Big\} \P\big[\tau\in[s_i,s_{i+1})\big]
	\Big).
\end{align}

We next involve the LDP of $ W_\lambda $ to control $ \P[\tau\in[s_i,s_{i+1})] $.
The Brownian bridge enjoys an LDP on $ C((t_0,x_0)\to(0,0)) $, the space of continuous paths connecting $ (t_0,x_0) $ and $ (0,0) $.
The rate function is $ I_\mathrm{BB}(\Path) := \frac{1}{2} \int_0^{t_0} \dot{\Path}^2 \, \d u - x_0^2/(2t_0) $, where the integral is interpreted as $ +\infty $ when $ \Path \notin H^1((t_0,x_0)\to(0,0)) $.
Now, consider the set $ \calT_i := \{ \Path: \Path[s_i,s_{i+1}] \cap (\partial \lens) \neq \emptyset, \, \Path[s_{i+1},t_0] \cap \lens^\circ = \emptyset \} $.
Indeed, $ \calT_i $ is a closed set in $ C((t_0,x_0)\to(0,0)) $ and $ \{\tau\in[s_i,s_{i+1})\} \subset \{ W_\lambda \in \calT_i \} $.
By the LDP of $ W_\lambda $,
\begin{align}
	\label{e.upbd.lensc.ldp}
	\limsup_{\lambda\to\infty} \frac{1}{\lambda} \log \P\big[\tau\in[s_i,s_{i+1})\big]
	\leq
	-\inf_{\Path\in\calT_i}  \frac{1}{2} \int_0^{t_0} \dot{\Path}^2 \, \d u  + \frac{x_0^2}{2t_0}.
\end{align}
We would like to modify the right side of \eqref{e.upbd.lensc.ldp}.
Let $ t_\Path := \sup\{ u : \Path(u) \leq \ell(u) \} $ be the first (going backward in time) hitting time of $ \Path \in H^1((t_0,x_0)\to(0,0)) $ and consider the set $ \calS_i := \{ \Path : t_\Path \in [s_i,s_{i+1}] \} $.
Even though $ \calS_i \subsetneq \calT_i $, it is not hard to verify that the infimum in \eqref{e.upbd.lensc.ldp} does not change if we replace $ \calT_i $ with $ \calS_i $.
Hence
\begin{align}
	\label{e.upbd.lensc.ldp.}
	\limsup_{\lambda\to\infty} \frac{1}{\lambda} \log \P\big[\tau\in[s_i,s_{i+1})\big]
	\leq
	-\inf\Big\{  \frac{1}{2} \int_0^{t_0} \dot{\Path}^2 \, \d u \, : \Path \in H^1((t_0,x_0)\to(0,0)), t_\Path \in [s_i,s_{i+1}] \Big\}  + \frac{x_0^2}{2t_0}.
\end{align}

We now combine \eqref{e.upbd.lensc.discrete} and \eqref{e.upbd.lensc.ldp.} to complete the proof.
Apply $ \limsup_{\lambda\to\infty}\frac{1}{\lambda}\log(\Cdot) $ to both sides of \eqref{e.upbd.lensc.discrete} with the aid of \eqref{e.upbd.lensc.ldp.}, and take $ n\to\infty $ with the aid of the continuity of $ \hlim(u,\ell(u))|_{u\in[0,t_0]} $.
We arrive at 
\begin{align}
	\limsup_{\lambda\to\infty} \hm_\lambda(t_0,x_0)
	\leq
	\sup\Big\{  -\int_0^{t_0} \frac{1}{2} \dot{\Path}^2 \, \d u +\frac{\ell(t_\Path)^2}{2t_\Path} + \hlim(t_\Path,\ell(t_\Path)) \, : \, \Path \in H^1((t_0,x_0)\to(0,0)) \Big\}.
\end{align}
Bound from above the first two terms within the supremum together by $ -\int_{t_\Path}^{t_0} \frac{1}{2} \dot{\Path}^2 \, \d u $.
To see why such a bound holds, write the first term as $ -\int_{t_\Path}^{t_0} - \int_0^{t_\Path} $, and note that the second term is $ \int_0^{t_\Path} \frac12 \dot{\Path}_\mathrm{line}^2 \, \d u $, where ${\Path}_\mathrm{line}$ is the linear path that connects $ (t_\Path,\ell(\Path)) $ and $ (0,0) $.
A linear path minimizes such an integral among all paths connecting the same points, so the bound follows.
Next, use \eqref{e.hlim.geodesic} to express $ \hlim(t_\Path,\ell(t_\Path)) $ as $ \int_{0}^{t_\gamma} (-\frac12 \dot\ell(u)^2 + \devlim(u,\ell(u)) )\, \d u $.
At this stage, the relevant paths are $ \ell(u)|_{u\in[0,t_\Path]} $ and $ \Path(u)|_{u\in[t_\Path,t_0]} $.
Concatenating these paths gives a path $ \perturbation\in H^1((t_0,x_0)\to(0,0)) $.
Hence,
\begin{align*}
	\limsup_{\lambda\to\infty} \hm_\lambda(t_0,x_0)
	\leq
	\sup\Big\{ \int_0^{t_0} \Big( -\frac{1}{2} \dot{ \perturbation}^2 + \devlim(u, \perturbation(u)) \Big) \, \d u  \, : \,  \perturbation \in H^1((t_0,x_0)\to(0,0)) \Big\}.
\end{align*}
The right side is exactly $ \hlim(t_0,x_0) $; see \eqref{e.hlim}.
This completes the proof.

\appendix

\section{Technical tools}
\label{s.a.tools}

\subsection{Some bounds}
Here are some useful bounds.
It is not hard to derive them from \eqref{e.r.ode}, \eqref{e.r.bound}, \eqref{e.devlim}, and \eqref{e.hlim.geodesic.0}.
\begin{align}
	\label{e.devlim.bd}
	\sup_{x\in\R} |\devlim(t,x)| &\leq c\, (t^{-2/3} + (2-t)^{-2/3}),
\\
	\label{e.elldot.bd}
	|\dot{\ell}(t)| &\leq c\, (t^{-1/3} + (2-t)^{-1/3}),
\\
	\label{e.hlim.bd}
	\sup_{(t,x)\in(0,2]\times\R} \big|\hlim(t,x) + \tfrac{x^2}{2t} \big| &<\infty,
\end{align}
where $ c \in(0,\infty) $ denotes a universal constant.

\subsection{The sign of $ \devm_\lambda $}
\label{s.a.devm<=0}
Next, we prove that $ \devm_\lambda \leq 0 $ a.e.\ on $ [0,2]\times\R $, for all $ e^{-\lambda} < p_\lambda(2,0) $.
Let us argue by contradiction: Consider $ (\devm_\lambda)_+ := \min\{\devm_\lambda,0\} $ and assume the function is not 0 a.e.
Set $ \dev^1 := \devm_\lambda - (\devm_\lambda)_+ $.
By the monotonicity of $ \FK_\lambda[\dev] $ in $ \dev $, we have $ e^{-\lambda} = \FK_\lambda[\devm_\lambda](2,0) \geq \FK_\lambda[\dev^1](2,0) $.
Also, we have $ \FK_\lambda[0](2,0) = \hk_\lambda(2,0) > e^{-\lambda} $.
It is not hard to verify that $ \FK_\lambda[a\dev^1](2,0) $ varies continuously in $ a $.
Hence there exists $ a_1\in(0,1] $ such that $ \FK_\lambda[a_1\dev^1](2,0) = e^{-\lambda} $.
However, we have $ \norm{\devm_\lambda}_2 > \norm{\dev^1}_2 \geq \norm{a_1\dev^1}_2 $, which contradicts with the assumption that $ \devm_\lambda $ is the minimizer.

\subsection{The Feynman--Kac formula}
\label{s.a.FKformula}
Finally, let us provide a proof of the Feynman--Kac formula~\eqref{e.FKformula}.
We consider $ \lambda=1 $ only, and the case for $ \lambda>0 $ follows by scaling.
Let $ \hk(t,x) $ denote the standard heat kernel.
First, by writing \eqref{e.he} in Duhamel's form and iterating the result, we have
\begin{align}
	\label{e.he.expansion}
	\Zfn[\dev](t,x) = \hk(t,x) + \sum_{n=1}^\infty I_n(t,x),
\end{align}
where
\begin{align}
	\label{e.she.In}
	I_n(t,x) 
	:= 
	\int_{0<u_1<\ldots<u_n<t} \Big( \int_{\R^n} 
		\hk(u_1,z_1)\prod_{i=1}^n \Big( \hk(u_{i+1}-u_i,z_{i+1}-z_i) \dev(u_{i},z_i) \d z_i \Big)
	\Big) \d u_1\cdots \d u_n,
\end{align}
with the convention $ u_{n+1}:=t $ and $ z_{n+1} := x $.
To ensure that \eqref{e.he.expansion} gives a convergence series, we derive a bound on $ I_n(t,x) $.
Apply the Cauchy--Schwarz inequality in \eqref{e.she.In} over the integrals to ``separate the $ \hk $ and $ \dev $'' to get
\begin{align}
	\label{e.cs.1}
	|I_n(t,x)|^2
	\leq&
	\int_{0<s_1<\ldots<s_n<t} \Big( \int_{\R^n} 
		\prod_{i=1}^n \Big( \dev(s_i,y_i)^2 \d y_i \Big)
	\Big) \d s_1\cdots \d s_n
\\
	\label{e.cs.2}
	&\cdot
	\int_{0<s_1<\ldots<s_n<t} \Big( \int_{\R^n} 
		\hk^2(s_1,y_1) \prod_{i=1}^n \Big( \hk^2(s_{i+1}-s_i,y_{i+1}-y_i) \d y_i \Big)
	\Big) \d s_1\cdots \d s_n.
\end{align}
The integral in \eqref{e.cs.1} is equal to $ \frac{1}{n!}\norm{\dev}^2_{L^2([0,t]\times\R)} $; the integral in \eqref{e.cs.2} is 
$  
	2^{-n} t^{n/2-1}\pi^{-1/2} e^{-x^2/t} / \Gamma(\frac{n+1}{2}).
$
Therefore,
\begin{align}
	\label{e.Inbd}
	|I_n(t,x)|
	\leq
	\pi^{1/4} \, (n! \,\Gamma(\tfrac{n+1}{2}))^{-1/2}\,
	\big( 2^{-1/2}t^{1/4}\norm{\dev}_{L^2([0,t]\times\R)} \big)^n
	\cdot
	\hk(t,x).
\end{align}
Next, recall that the Brownian Bridge $ W\sim\mathrm{BB}_1((t,x)\to(0,0)) $ has the following finite dimensional distributions; see~\cite[6.11~Problem]{karatzas12} for example.
\begin{align*}
	\P\Big[W(u_1)\in A_1, W(u_2)\in A_2, \ldots,W(u_n) \in A_n\Big]
	=
	\frac{ 1 }{ \hk(t,x) } \displaystyle \int_{\prod_{i=1}^n A_i} \hk(u_1,z_1)\prod_{i=1}^n \Big( \hk(u_{i+1}-u_i,z_{i+1}-z_i) \d z_i \Big),
\end{align*}
with the convention $ u_0:=0 $, $ u_{n+1}:=t $, $ z_0:=0 $, and $ z_{n+1}:=x $.
We hence recognize $ I_n(t,x) $ as an expectation over the Brownian bridge:
\begin{align}
	\label{e.matching0}
	I_n(t,x)
	&=
	\int_{0<u_1<\ldots<u_n<t} \E\Big[ \Big( \prod_{i=1}^n \dev(u_i,W(u_i)) \Big) \Big] \cdot \hk(t,x) \, \d u_1 \cdots \d u_n
\\
	\label{e.matching1}
	&=
	\E\Big[\int_{0<u_1<\ldots<u_n<t}  \Big( \prod_{i=1}^n \dev(u_i,W(u_i)) \Big)\, \d u_1 \cdots \d u_n \Big] \cdot \hk(t,x).
\end{align}
The exchange of the sum and integral in \eqref{e.matching0}--\eqref{e.matching1} is justified: If we replace $ \dev $ with $ |\dev| $ in \eqref{e.matching0}--\eqref{e.matching1}, the result is bounded by the right hand side of \eqref{e.Inbd}; hence Fubini's theorem applies.
Next, the integral in \eqref{e.matching1} can be written as
\begin{align}
	\label{e.matching2}
	\frac{1}{n!} \int_{0}^{t}\cdots \int_{0}^t \Big( \prod_{i=1}^n \dev(u'_i,W(u'_i)) \Big)\, \d u'_1 \cdots \d u'_n
	=
	\frac{1}{n!} \Big( \int_{0}^{t}\dev(u',W(u'))\, \d u' \Big)^n.
\end{align}
This is because $ \prod_{i=1}^n \dev(u'_i,W(u'_i)) $ is invariant under reordering of the $ u'_i $.
Now, combine~\eqref{e.matching0}, \eqref{e.matching1}, and \eqref{e.matching2}; sum the result over $ n \geq 1 $; add $ \hk(t,x) $ to the result; combine the result with \eqref{e.he.expansion}.
We have
\begin{align*}
	\Zfn[\dev](t,x)
	=
	\sum_{n=0}^\infty \E\Big[ \frac{1}{n!} \Big( \int_{0}^{t}\dev(u,W(u)) \, \d u \Big)^n \Big] \cdot \hk(t,x)
	=
	\E\Big[ \exp\Big(\int_0^t \dev(u,W(u))\,\d u \Big) \Big] \cdot \hk(t,x).
\end{align*}
The right side is exactly $ \FK_1[\dev](t,x) $.

\bibliographystyle{alphaabbr}
\bibliography{wnt}

\newcommand{\etalchar}[1]{$^{#1}$}
\begin{thebibliography}{HLDM{\etalchar{+}}18}

\bibitem[ALM19]{asida19}
T.~Asida, E.~Livne, and B.~Meerson.
\newblock Large fluctuations of a {K}ardar{--P}arisi{--Z}hang interface on a
  half line: The height statistics at a shifted point.
\newblock {\em Phys Rev E}, 99(4):042132, 2019.

\bibitem[BSM22]{bettelheim22}
E.~Bettelheim, N.~R. Smith, and B.~Meerson.
\newblock Inverse scattering method solves the problem of full statistics of
  nonstationary heat transfer in the {K}ipnis{--M}archioro{--P}resutti model.
\newblock {\em Phys Rev Lett}, 128(13):130602, 2022.

\bibitem[Car10]{cardaliaguet10}
P.~Cardaliaguet.
\newblock Notes on mean field games.
\newblock Technical report, 2010.

\bibitem[CC22]{cafasso22}
M.~Cafasso and T.~Claeys.
\newblock A {R}iemann{--H}ilbert approach to the lower tail of the
  {K}ardar{--P}arisi{--Z}hang equation.
\newblock {\em Comm Pure Appl Math}, 75(3):493--540, 2022.

\bibitem[CG20a]{corwin20general}
I.~Corwin and P.~Ghosal.
\newblock {KPZ} equation tails for general initial data.
\newblock {\em Electron J Probab}, 25, 2020.

\bibitem[CG20b]{corwin20lower}
I.~Corwin and P.~Ghosal.
\newblock Lower tail of the {KPZ} equation.
\newblock {\em Duke Math J}, 169(7):1329--1395, 2020.

\bibitem[CGK{\etalchar{+}}18]{corwin18}
I.~Corwin, P.~Ghosal, A.~Krajenbrink, P.~Le~Doussal, and L.-C. Tsai.
\newblock {C}oulomb-gas electrostatics controls large fluctuations of the
  {K}ardar{--P}arisi{--Z}hang equation.
\newblock {\em Phys Rev Lett}, 121(6):060201, 2018.

\bibitem[DT21]{das21}
S.~Das and L.-C. Tsai.
\newblock Fractional moments of the stochastic heat equation.
\newblock {\em Annales de l'Institut Henri Poincar{\'e}, Probabilit{\'e}s et
  Statistiques}, 57(2):778--799, 2021.

\bibitem[GH22]{ganguly22}
S.~Ganguly and M.~Hegde.
\newblock Sharp upper tail estimates and limit shapes for the {KPZ} equation
  via the tangent method.
\newblock {\em arXiv:2208.08922}, 2022.

\bibitem[GL22]{ghosal20}
P.~Ghosal and Y.~Lin.
\newblock Lyapunov exponents of the {SHE} for general initial data.
\newblock {\em To appear in Annales de l'Institut Henri Poincar{\'e},
  Probabilit{\'e}s et Statistiques. arXiv:2007.06505 (2020)}, 2022+.

\bibitem[GLL11]{gueant11}
O.~Gu{\'e}ant, J.-M. Lasry, and P.-L. Lions.
\newblock Mean field games and applications.
\newblock In {\em Paris-Princeton lectures on mathematical finance 2010}, pages
  205--266. Springer, 2011.

\bibitem[GLLT21]{gaudreaulamarre21}
P.~Y. Gaudreau~Lamarre, Y.~Lin, and L.-C. Tsai.
\newblock {KPZ} equation with a small noise, deep upper tail and limit shape.
\newblock {\em arXiv:2106.13313}, 2021.

\bibitem[HKLD20]{hartmann20}
A.~K. Hartmann, A.~Krajenbrink, and P.~Le~Doussal.
\newblock Probing large deviations of the {K}ardar{--P}arisi{--Z}hang equation
  at short times with an importance sampling of directed polymers in random
  media.
\newblock {\em Phys Rev E}, 101(1):012134, 2020.

\bibitem[HLDM{\etalchar{+}}18]{hartmann18}
A.~K. Hartmann, P.~Le~Doussal, S.~N. Majumdar, A.~Rosso, and G.~Schehr.
\newblock High-precision simulation of the height distribution for the {KPZ}
  equation.
\newblock {\em EPL (Europhysics Letters)}, 121(6):67004, 2018.

\bibitem[HMS19]{hartmann19}
A.~K. Hartmann, B.~Meerson, and P.~Sasorov.
\newblock Optimal paths of nonequilibrium stochastic fields: The
  {K}ardar{-P}arisi{-Z}hang interface as a test case.
\newblock {\em Physical Review Research}, 1(3):032043, 2019.

\bibitem[HMS21]{hartmann21}
A.~K. Hartmann, B.~Meerson, and P.~Sasorov.
\newblock Observing symmetry-broken optimal paths of the stationary
  {K}ardar{--P}arisi{--Z}hang interface via a large-deviation sampling of
  directed polymers in random media.
\newblock {\em Phys Rev E}, 104(5):054125, 2021.

\bibitem[JKM16]{janas16}
M.~Janas, A.~Kamenev, and B.~Meerson.
\newblock Dynamical phase transition in large-deviation statistics of the
  {K}ardar{--P}arisi{--Z}hang equation.
\newblock {\em Phys Rev E}, 94(3):032133, 2016.

\bibitem[Kim21]{kim21}
Y.~H. Kim.
\newblock The lower tail of the half-space kpz equation.
\newblock {\em Stochastic Processes and their Applications}, 142:365--406,
  2021.

\bibitem[KK07]{kolokolov07}
I.~Kolokolov and S.~Korshunov.
\newblock Optimal fluctuation approach to a directed polymer in a random
  medium.
\newblock {\em Phys Rev B}, 75(14):140201, 2007.

\bibitem[KK08]{kolokolov08}
I.~Kolokolov and S.~Korshunov.
\newblock Universal and nonuniversal tails of distribution functions in the
  directed polymer and {K}ardar{--P}arisi{--Z}hang problems.
\newblock {\em Phys Rev B}, 78(2):024206, 2008.

\bibitem[KK09]{kolokolov09}
I.~Kolokolov and S.~Korshunov.
\newblock Explicit solution of the optimal fluctuation problem for an elastic
  string in a random medium.
\newblock {\em Phys Rev E}, 80(3):031107, 2009.

\bibitem[KLD17]{krajenbrink17short}
A.~Krajenbrink and P.~Le~Doussal.
\newblock Exact short-time height distribution in the one-dimensional
  {K}ardar{--P}arisi{--Z}hang equation with {B}rownian initial condition.
\newblock {\em Phys Rev E}, 96(2):020102, 2017.

\bibitem[KLD18a]{krajenbrink18half}
A.~Krajenbrink and P.~Le~Doussal.
\newblock Large fluctuations of the {KPZ} equation in a half-space.
\newblock {\em SciPost Phys}, 5:032, 2018.

\bibitem[KLD18b]{krajenbrink18simple}
A.~Krajenbrink and P.~Le~Doussal.
\newblock Simple derivation of the {$(-\lambda H)^{5/2}$} tail for the 1{D}
  {KPZ} equation.
\newblock {\em J Stat Mech Theory Exp}, 2018(6):063210, 2018.

\bibitem[KLD19]{krajenbrink19}
A.~Krajenbrink and P.~Le~Doussal.
\newblock Linear statistics and pushed {C}oulomb gas at the edge of
  $\beta$-random matrices: Four paths to large deviations.
\newblock {\em EPL (Europhysics Letters)}, 125(2):20009, 2019.

\bibitem[KLD21]{krajenbrink21}
A.~Krajenbrink and P.~Le~Doussal.
\newblock Inverse scattering of the {Z}akharov{--S}habat system solves the weak
  noise theory of the {K}ardar{--P}arisi{--Z}hang equation.
\newblock {\em Phys Rev Lett}, 127(6):064101, 2021.

\bibitem[KLD22a]{krajenbrink22}
A.~Krajenbrink and P.~Le~Doussal.
\newblock The crossover from the {M}acroscopic {F}luctuation {T}heory to the
  {K}ardar{--P}arisi{--Z}hang equation controls the large deviations beyond
  {E}instein's diffusion.
\newblock {\em arXiv:2204.04720}, 2022.

\bibitem[KLD22b]{krajenbrink21flat}
A.~Krajenbrink and P.~Le~Doussal.
\newblock Inverse scattering solution of the weak noise theory of the
  {K}ardar{--P}arisi{--Z}hang equation with flat and {B}rownian initial
  conditions.
\newblock {\em Phys Rev E}, 105:054142, 2022.

\bibitem[KLDP18]{krajenbrink18systematic}
A.~Krajenbrink, P.~Le~Doussal, and S.~Prolhac.
\newblock Systematic time expansion for the {K}ardar{--P}arisi{--Z}hang
  equation, linear statistics of the {GUE} at the edge and trapped fermions.
\newblock {\em Nucl Phys B}, 936:239--305, 2018.

\bibitem[KMS16]{kamenev16}
A.~Kamenev, B.~Meerson, and P.~V. Sasorov.
\newblock Short-time height distribution in the one-dimensional
  {K}ardar{--P}arisi{--Z}hang equation: Starting from a parabola.
\newblock {\em Phys Rev E}, 94(3):032108, 2016.

\bibitem[Kra19]{krajenbrink19thesis}
A.~Krajenbrink.
\newblock {\em Beyond the typical fluctuations: a journey to the large
  deviations in the {K}ardar{-P}arisi{-Z}hang growth model}.
\newblock PhD thesis, PSL Research University, 2019.

\bibitem[Kra20]{krajenbrink20det}
A.~Krajenbrink.
\newblock From {P}ainlev{\'e} to {Z}akharov{--S}habat and beyond: Fredholm
  determinants and integro-differential hierarchies.
\newblock {\em J Phys Math Theor}, 54(3):035001, 2020.

\bibitem[KS12]{karatzas12}
I.~Karatzas and S.~Shreve.
\newblock {\em Brownian motion and stochastic calculus}, volume 113.
\newblock Springer Science \& Business Media, 2012.

\bibitem[LD20]{ledoussal19kp}
P.~Le~Doussal.
\newblock Large deviations for the {K}ardar{--P}arisi{--Z}hang equation from
  the {K}adomtsev{--P}etviashvili equation.
\newblock {\em J Stat Mech Theory Exp}, 2020.
\newblock 043201.

\bibitem[LDMRS16]{ledoussal16short}
P.~Le~Doussal, S.~N. Majumdar, A.~Rosso, and G.~Schehr.
\newblock Exact short-time height distribution in the one-dimensional
  {K}ardar{--P}arisi{--Z}hang equation and edge fermions at high temperature.
\newblock {\em Phys Rev Lett}, 117(7):070403, 2016.

\bibitem[LDMS16]{ledoussal16long}
P.~Le~Doussal, S.~N. Majumdar, and G.~Schehr.
\newblock Large deviations for the height in 1{D} {K}ardar-{P}arisi-{Z}hang
  growth at late times.
\newblock {\em EPL (Europhysics Letters)}, 113(6):60004, 2016.

\bibitem[Lin21]{lin21half}
Y.~Lin.
\newblock {L}yapunov exponents of the half-line {SHE}.
\newblock {\em J Stat Phys}, 183(3):1--34, 2021.

\bibitem[Lio07]{lions07}
P.~Lions.
\newblock College de france course on mean-field games.
\newblock {\em College de France}, 2011, 2007.

\bibitem[LT21]{lin21}
Y.~Lin and L.-C. Tsai.
\newblock Short time large deviations of the {KPZ} equation.
\newblock {\em Commun Math Phys}, 386(1):359--393, 2021.

\bibitem[LW22]{liu22}
Z.~Liu and Y.~Wang.
\newblock A conditional scaling limit of the {KPZ} fixed point with height
  tending to infinity at one location.
\newblock {\em arXiv:2208.12215}, 2022.

\bibitem[MKV16]{meerson16}
B.~Meerson, E.~Katzav, and A.~Vilenkin.
\newblock Large deviations of surface height in the {K}ardar{--P}arisi{--Z}hang
  equation.
\newblock {\em Phys Rev Lett}, 116(7):070601, 2016.

\bibitem[MMS22]{mallick22}
K.~Mallick, H.~Moriya, and T.~Sasamoto.
\newblock Exact solution of the macroscopic fluctuation theory for the
  symmetric exclusion process.
\newblock {\em arXiv:2202.05213}, 2022.

\bibitem[MS17]{meerson17}
B.~Meerson and J.~Schmidt.
\newblock Height distribution tails in the {K}ardar{--P}arisi{--Z}hang equation
  with brownian initial conditions.
\newblock {\em J Stat Mech Theory Exp}, 2017(10):103207, 2017.

\bibitem[MV18]{meerson18}
B.~Meerson and A.~Vilenkin.
\newblock Large fluctuations of a {K}ardar{-P}arisi{-Z}hang interface on a half
  line.
\newblock {\em Physical Review E}, 98(3):032145, 2018.

\bibitem[SM18]{smith18exact}
N.~R. Smith and B.~Meerson.
\newblock Exact short-time height distribution for the flat
  {K}ardar{--P}arisi{--Z}hang interface.
\newblock {\em Phys Rev E}, 97(5):052110, 2018.

\bibitem[SMP17]{sasorov17}
P.~Sasorov, B.~Meerson, and S.~Prolhac.
\newblock Large deviations of surface height in the 1+1-dimensional
  {K}ardar{--P}arisi{--Z}hang equation: exact long-time results for ${\lambda}
  h< 0$.
\newblock {\em J Stat Mech Theory Exp}, 2017(6):063203, 2017.

\bibitem[SMS18]{smith18finitesize}
N.~R. Smith, B.~Meerson, and P.~Sasorov.
\newblock Finite-size effects in the short-time height distribution of the
  {K}ardar{--P}arisi{--Z}hang equation.
\newblock {\em J Stat Mech Theory Exp}, 2018(2):023202, 2018.

\bibitem[SMV19]{smith19}
N.~R. Smith, B.~Meerson, and A.~Vilenkin.
\newblock Time-averaged height distribution of the {K}ardar{--P}arisi{--Z}hang
  interface.
\newblock {\em J Stat Mech Theory Exp}, 2019(5):053207, 2019.

\bibitem[Tsa22a]{tsai22}
L.-C. Tsai.
\newblock Exact lower-tail large deviations of the kpz equation.
\newblock {\em Duke Math J}, 1(1):1--44, 2022.

\bibitem[Tsa22b]{tsai22wnt}
L.-C. Tsai.
\newblock Integrability in the weak noise theory.
\newblock {\em arXiv:2204.00614}, 2022.

\end{thebibliography}
\end{document}